\documentclass{amsart}
\usepackage{amsmath, amstext, amsbsy, amssymb}

\hoffset \voffset \oddsidemargin=55pt \evensidemargin=55pt
\numberwithin{equation}{section}
\def\beq{\begin{eqnarray}}
\def\eeq{\end{eqnarray}}
\def\beqs{\begin{eqnarray*}}
\def\eeqs{\end{eqnarray*}}

\def\ind{\hbox{\rm ind}}
\def\ord{\hbox{\rm ord}}

\newfont{\df}{eufm10}

\def\ct{{\mathcal T}}

\title[On the index-Conjecture on length four sequences]
{On the index-Conjecture on length four minimal zero-sum sequences}
\thanks{the author's email: xialimeng@ujs.edu.cn}
\thanks{Supported by the  NNSF of China (Grant No. 11001110, 11271131)}
\author[L.-M. Xia]{Li-meng   Xia}

\date{}
\begin{document}
\maketitle
\centerline{Faculty of Science, Jiangsu University}
\centerline{Zhenjiang, 212013, Jiangsu Province, P.R. China}

\def\abstractname{ABSTRACT}
\begin{abstract}

Let $G$ be a finite cyclic group. Every sequence $S$ over $G$ can be written in the form
$S=(n_1g)\cdot...\cdot(n_lg)$ where $g\in G$ and $n_1,\cdots,n_l\in[1,{\hbox{\rm ord}}(g)]$, and the index $\ind(S)$ of $S$ is defined to be the minimum of $(n_1+\cdots+n_l)/\hbox{\rm ord}(g)$ over all possible $g\in G$ such that $\langle g\rangle=G$. A conjecture says that if $G$ is finite such that $\gcd(|G|,6)=1$, then $\ind(S)=1$ for every minimal zero-sum sequence $S$. In this paper, we prove that the conjecture holds if $S$ is reduced and at least one $n_i$ coprime to $|G|$.


\end{abstract}

\newtheorem{theo}{Theorem}[section]
\newtheorem{theorem}[theo]{Theorem}
\newtheorem{defi}[theo]{Definition}
\newtheorem{conj}[theo]{Conjecture}
\newtheorem{lemma}[theo]{Lemma}
\newtheorem{coro}[theo]{Corollary}
\newtheorem{proposition}[theo]{Proposition}
\newtheorem{remark}[theo]{Remark}

\setcounter{section}{0}

\section{Introduction}

Throughout the paper, let $G$ be an additively written finite cyclic group of order $|G|=n$. By a sequence over $G$ we mean a finite sequence of terms from $G$ which is unordered and repetition of terms is allowed. We view sequences over $G$ as elements of the free abelian monoid $\mathcal{F}(G)$ and use multiplicative notation. Thus a sequence $S$ of length $|S|=k$ is written in the form $S=(n_1g)\cdot...\cdot(n_kg)$, where $n_1,\cdots,n_k\in\mathbb{N}$ and $g\in G$. We call $S$ a {\it zero-sum sequence} if $\sum_{j=1}^kn_jg=0$. If $S$ is a zero-sum sequence, but no proper nontrivial subsequence of $S$ has sum zero, then $S$ is called a {\it minimal zero-sum sequence}. Recall that the index of a sequence $S$ over $G$ is defined as follows.

\begin{defi}
For a sequence over $G$
\beqs S=(n_1g)\cdot...\cdot(n_kg), &&\hbox{where}\;1\leq n_1,\cdots,n_k\leq n,\eeqs
the index of $S$ is define by $\ind(S)=\min\{\|S\|_g|g\in G \hbox{~with~}\langle g\rangle=G\}$, where
\beqs \|S\|_g=\frac{n_1+\cdots+n_k}{\ord(g)}.\eeqs
\end{defi}
Clearly, $S$ has sum zero if and only if $\ind(S)$ is an integer.

\begin{conj}
 Let $G$ be a finite cyclic group such that $\gcd(|G|,6)=1$. Then every minimal zero-sum sequence $S$ over $G$ of length $|S|=4$ has $\ind(S)=1$.
\end{conj}

The index of a sequence is a crucial invariant in the investigation of (minimal) zero-sum
sequences (resp. of zero-sum free sequences) over cyclic groups. It was first addressed by
Kleitman-Lemke (in the conjecture [9, page 344]), then used as a key tool by Geroldinger ([6, page736]), and investigated by Gao [3] in a systematical way. Since then it has received a great
deal of attention (see for example [1, 2, 4, 7, 10, 11, 12, 13, 14, 15, 16, 17, 18]). A main focus of the investigation of index is to determine minimal zero-sum sequences of index 1. If $S$ is a minimal zero-sum sequence of length $|S|$ such that $|S|\leq3$ or $|S|\geq\lfloor \frac{n}2\rfloor+2$, then $\ind(S)=1$ (see [1, 14, 16]). In contrast to that, it was shown that for each $k$ with $5\leq k\leq \lfloor \frac{n}2\rfloor+1$, there is a minimal zero-sum subsequence $T$ of length $|T| = k$ with $\ind(T)\geq 2$ ([13, 15]) and that the same is true for $k = 4$ and $\gcd(n, 6)\not= 1$ ([13]). The left case leads to the above conjecture.

In [12], it was prove that Conjecture 1.2 holds true if $n$ is a prime power. In [11], it was prove that Conjecture 1.2 holds for $n=p_1^\alpha\cdot p_2^\beta,(p_1\not=p_2)$ and if the sequence contains an element $g$ of order $\ord(g)=n$. However, the general case is still open.

\begin{defi} Let $S=(n_1g)\cdot...\cdot(n_kg)$ be a minimal zero-sum  sequence over $G$.
Then $S$ is called reduced if $(pn_1g)\cdot...\cdot(pn_kg)$ is not a minimal zero-sum sequence for any prime factor $p$ of $n$.
\end{defi}

In this paper, our main result is stated by the following theorem.

\begin{theo}
Let $G$ be a  finite cyclic group such that $\gcd(|G|,6)=1$, $S=(n_1g)\cdot...\cdot(n_kg)$ be a minimal zero-sum  sequence over $G$ with $\ord(g)=|G|$. If $S$ is reduced and at least one $n_i$ coprime to $n$, then $\ind(S)=1$
\end{theo}

It was mentioned in \cite{P} that Conjecture 1.2 was confirmed computationally if $n\leq1000$. Hence, throughout the paper, we always assume that $n>1000$.

\section{Induction on prime decomposition of $n$}

Throughout, let $G$ be a cyclic group of order $|G|=n>1000$. Given real numbers $a,b\in\mathbb{R}$, we use $[a,b]=\{x\in\mathbb{Z}|a\leq x\leq b\}$ to denote the set of integers between $a$ and $b$. For $x\in\mathbb{Z}$, we denote by $|x|_n\in[1,n]$ the integer congruent to $x$ modulo $n$. Suppose that $n$ has a prime decomposition $n=p_1^{\mu_1}\cdots p_d^{\mu_d}$. Let $S=(x_1g)\cdot...\cdot(x_kg)$  be a minimal zero-sum sequence over $G$ such that $\ord(g)=n=|G|$ and $1\leq x_1,x_2,x_3,x_4\leq n-1$. Then $x_1+x_2+x_3+x_4=\nu n$, where $1\leq\nu\leq 3$.

For convenience, we use the following symbols:
\beqs \ct=\{p_1,\cdots, p_d\},&& \ct_i=\{p\in\ct|p=\gcd(p,x_i)\},\;i=1,2,3,4.\eeqs

\begin{theo}
If $S$ is reduced and $\gcd(x_1,x_2,x_3,x_4,n)=1$, then $|\ct|\leq3$. Particularly, if $|\ct|=3$, then after renumbering if necessary one of the following statements holds:

(A1) $\{\gcd(x_i,n)|i=1,2,3,4\}=\{p_1p_2,p_2,p_1p_3,p_3\}$.

(A2) $\{\gcd(x_i,n)|i=1,2,3,4\}=\{1,p_1,p_2,p_1p_2\}$.

(A3) $\gcd(x_i,n)=1$ for $i=1,2,3,4$.

(A4) $\gcd(x_1,n)=1,\gcd(x_2,n)=p_1p_2,\gcd(x_3,n)=p_1p_3,\gcd(x_4,n)=p_2p_3$.

\end{theo}

For the proof of this theorem, we need the following lemma.

\begin{lemma}
Suppose that $|\ct|\geq3$, $p\in\ct$ and $1\leq|px_i|_n\leq n-1$ for $i=1,2,3,4$. If for any $q\in\ct$, $(qx_1g)\cdot(qx_2g)\cdot(qx_3g)\cdot(qx_4g)$ is not a minimal zero-sum sequence, then $n=p_1p_2p_3$ and one of (A1),
(A2), (A3) holds.

Particularly, we can assume that $x_1=1, \{\gcd(n,x_2),\gcd(n,x_3),\gcd(n,x_4)\}=\{p_1, p_2, p_1p_2\}$ for (A2), and $x_1=1, p_1p_2|(x_2+1),p_1p_3|(x_3+1),p_2p_3|(x_4+1)$ for (A3).
\end{lemma}
\begin{proof}

Since  $(px_1g)\cdot(px_2g)\cdot(px_3g)\cdot(px_4g)$ is not a minimal zero-sum sequence, without loss of generality, we can assume that $|px_1|_n+|px_2|_n=n$ and $|px_3|_n+|px_4|_n=n$. We distinguish four cases.

{\bf Case 1.} $p\in\ct_1\cap\ct_2$.

For any $q\in\ct_3$, $(qx_1g)\cdot(qx_2g)\cdot(qx_3g)\cdot(qx_4g)$ is a minimal zero-sum sequence, hence $\ct_3=\emptyset$.

If $|\ct_1|>2$, then there is $q\in\ct_1$ such that $(qx_1g)\cdot(qx_2g)\cdot(qx_3g)\cdot(qx_4g)$ is a minimal zero-sum sequence, contradiction..

If $|\ct_1|=2<|\ct|$, then there is $q\in\ct_1$ such that $(qx_1g)\cdot(qx_2g)\cdot(qx_3g)\cdot(qx_4g)$ is a minimal zero-sum sequence, contradiction.

If $|\ct_1|=1, |\ct|>2$, then there is $q\in\ct$ such that $(qx_1g)\cdot(qx_2g)\cdot(qx_3g)\cdot(qx_4g)$ is a minimal zero-sum sequence, contradiction.

{\bf Case 2.} $p\in\ct_1\cap\ct_3$.

We must have $\gcd(x_2,n)|\gcd(x_1,n)$ and $\gcd(x_4,n)|\gcd(x_3,n)$. If $|\ct|\geq|\ct_2\cup\ct_3|+2$, then for $q\in\ct\setminus(\ct_2\cup\ct_3)$, $(qx_1g)\cdot(qx_2g)\cdot(qx_3g)\cdot(qx_4g)$ is a minimal zero-sum sequence. Since $|\ct|\geq3$, we have $\ct_2\cup\ct_4\not=\emptyset$.

If $|\ct_2\cup\ct_4|\geq3$, then there is $q\in\ct_2\cup\ct_4$ such that $(qx_1g)\cdot(qx_2g)\cdot(qx_3g)\cdot(qx_4g)$ is a minimal zero-sum sequence.

If $|\ct|=|\ct_2\cup\ct_3|+1$ and $|\ct_2\cup\ct_4|=2$, then for $q\in\ct_2\cup\ct_4$, $(qx_1g)\cdot(qx_2g)\cdot(qx_3g)\cdot(qx_4g)$ is a minimal zero-sum sequence.

If $|\ct|=|\ct_2\cup\ct_3|$ and $|\ct_2|=2$, then there is $q\in\ct_2$ such that $(qx_1g)\cdot(qx_2g)\cdot(qx_3g)\cdot(qx_4g)$ is a minimal zero-sum sequence.

If $|\ct|=|\ct_2\cup\ct_3|$ and $|\ct_2|=|\ct_4|=1$, then we can assume that $\ct_1=\{p_1,p_2\}, \ct_2 =\{p_2\}, \ct_3=\{p_1,p_3\}, \ct_4=\{p_3\}$.  $|p_1x_1|_n+|p_1x_2|_n=n$ implies that $\mu_1=1$.
Since $(p_2x_1g)\cdot(p_2x_2g)\cdot(p_2x_3g)\cdot(p_2x_4g)$ is not a minimal zero-sum sequence, we can get $\mu_2=1$. Similarly, $\mu_3=1$.

Besides all of above, we can assume $\ct=\{p_1,p_2,p_3\}, \ct_1=\{p_1,p_2\},\ct_2=\{p_1\}, \ct_3=\{p_2\}, \ct_4=\emptyset$. Moreover, $p_2^{\mu_2}|(p_2x_1+p_2x_2)$ implies $\mu_2=1$. Similarly, $p_1=1$. If $\mu_3>1$, then  it is easy to check that $(p_3x_1g)\cdot(p_3x_2g)\cdot(p_3x_3g)\cdot(p_3x_4g)$ is a minimal zero-sum sequence, contradiction. Hence $\mu_3=1$ and $n=p_1p_2p_3$.

{\bf Case 3.} $p\in\ct_1, p\not\in\cap_{i=2}^4\ct_i$.

We must have $\gcd(x_2,n)|\gcd(x_1,n)$ and $\gcd(x_3,n)=\gcd(x_4,n)$.

If $\ct\not=\ct_1\cup\ct_3$ or $|\ct_3|\geq2$, then for any $q\in\ct_3$, $(qx_1g)\cdot(qx_2g)\cdot(qx_3g)\cdot(qx_4g)$ is a minimal zero-sum sequence.

Let $\ct=\ct_1\cup\ct_3$ If $|\ct_3|=1$. For any $q\in\ct_2$, $(qx_1g)\cdot(qx_2g)\cdot(qx_3g)\cdot(qx_4g)$ is a minimal zero-sum sequence.

If $|\ct_2|\geq2$, then there is $q\in\ct_2$, such that $(qx_1g)\cdot(qx_2g)\cdot(qx_3g)\cdot(qx_4g)$ is a minimal zero-sum sequence.

{\bf Case 4.} $p\not\in\cup_{i=1}^4\ct_i$.

We must have $\gcd(x_1,n)=\gcd(x_2,n)$ and $\gcd(x_3,n)=\gcd(x_4,n)$. For any $q\in\ct_1$, it holds that $1\leq|qx_i|_n\leq n-1$ for $i=1,2,3,4$. If $\ct_3$ is not empty, then $|qx_1|_n+|qx_3|_n=n$ or  $|qx_2|_n+|qx_3|_n=n$. However, there is $q'\in\ct_3$ such that $q'\nmid qx_1,q'\nmid qx_2, q'\mid qx_3$, it is a contradiction. Repeat some similar discussions, we infer that  $|\ct_1|+|\ct_3|\leq1$.

If $\ct_1=\{q\}$, then there is $p'\in\ct\setminus\ct_1$. Clearly, $1\leq|p'x_i|_n\leq n-1$ for $i=1,2,3,4$. Then $|p'x_1|_n+|p'x_3|_n=n$ or  $|p'x_1|_n+|p'x_4|_n=n$. However, we have $q\mid p'x_1, q\nmid p'x_3, q\nmid p'x_4$, it is a contradiction.

If $\ct_1=\ct_3=\emptyset$. Then there exist $p_1,p_2,p_3\in\ct$ such that
\beqs \gcd(x_1+x_2,n)=\frac{n}{p_1}=\gcd(x_3+x_4,n),\\
\gcd(x_1+x_3,n)=\frac{n}{p_2}=\gcd(x_2+x_4,n),\\
\gcd(x_1+x_4,n)=\frac{n}{p_3}=\gcd(x_2+x_3,n).\eeqs

For any $q\in\ct\setminus\{p_1,p_2,p_3\}$, $(qx_1g)\cdot(qx_2g)\cdot(qx_3g)\cdot(qx_4g)$ must be a minimal zero-sum sequence, hence $\ct=\{p_1,p_2,p_3\}$.

If $\mu_1>1$, then $p_1|(x_1+x_2), p_1|(x_1+x_3), p_1|(x_2+x_3)$, we infer that $p_1|\gcd(x_1,x_2,x_3)$, contradiction. So $\mu_1=1$. Similarly, $\mu_2=\mu_3=1$.
\end{proof}

Now we can finish the proof of Theorem 2.1 via the following discussion:

(1) If $\cup_{i=1}^4\ct_i\not=\ct$, then for any $p\in \cup_{i=1}^4\ct_i$, we have
$1\leq|px_i|_n\leq n-1$ for $i=1,2,3,4$.

(2) If $\cup_{i=1}^4\ct_i$ is empty and $d\geq 2$, then for any $p\in \ct$, we have
$1\leq|px_i|_n\leq n-1$  for $i=1,2,3,4$.

For the above two cases, by Lemma 2.2, we can assume that $\cup_{i=1}^4\ct_i=\ct$ and $d\geq 3$. Without lace of generality, we let $x_1,x_2,x_3,x_4$ be such that
$|\ct_1|\leq|\ct_2|\leq|\ct_3|\leq|\ct_4|$.

(3) If $|\ct_3|\leq\frac{d}{2}$ and $\ct_4<d$, then for any $p\in \ct_4$, we have
$1\leq|px_i|_n\leq n-1$  for $i=1,2,3,4$.

(4) If $|\ct_3|\leq\frac{d}{2}$ and $\ct_4=d$. There must be an index $1\leq k\leq d$ such that $p_k^{\mu_k}\nmid x_4$. Then for any $j\not=k$, we have
$|p_jx_i|_n\leq n-1$  for $i=1,2,3,4$.

Now we can assume that $|\ct_3|>\frac{d}{2}$, then $\ct_3\cap\ct_4$ is nonempty(since $|\ct_3|+|\ct_4|\geq2|\ct_3|>d=|\ct|$).

(5) If $|\ct_3\cap\ct_4|\geq3$, there is $p\in\ct_3\cap\ct_4$ such that $1\leq|px_i|_n\leq n-1$  for $i=1,2,3,4$.

Clearly, there is $p\in\ct_3\cap\ct_4$ such that $1\leq|px_3|_n\leq n-1, 1\leq|px_4|_n\leq n-1$. If $n|px_2$, then for any $q(\not=p)\in\ct_3\cap\ct_4$, we have $q|x_2$, and then $q|x_1$, which contradicts to $\gcd(x_1,x_2,x_3,x_4,n)=1$. Hence $1\leq|px_2|_n\leq n-1$. Similarly, $1\leq|px_1|_n\leq n-1$.

(6) If $|\ct_3\cap\ct_4|=2$ and there is $p_k\in \ct_3\cap\ct_4$ such that $\mu_k\geq2$, $p_k^{\mu_k}\mid\gcd(x_3,x_4)$, then we have $1\leq|p_kx_i|_n\leq n-1$  for $i=1,2,3,4$.

It can be shown by similar argument in (5).

(7) If $\ct_3\cap\ct_4=\{p_k,p_l\}$ and $p_k^{\mu_k}\nmid x_3$, $p_l^{\mu_l}\nmid x_4$, then for any $j\not=k,l$, we have $1\leq|p_jx_i|_n\leq n-1$  for $i=1,2,3,4$.

(8) If $\ct_3\cap\ct_4=\{p_k,p_l\}$ and $p_k^{\mu_k}\nmid x_3$, $p_l^{\mu_l}\nmid x_3$, $p_k^{\mu_k}\mid x_4$, $p_l^{\mu_l}\mid x_4$, then we have $1\leq|p_kx_i|_n\leq n-1$  for $i=1,2,3,4$.

From (5),(6),(7),(8), we can assume that $|\ct_3\cap\ct_4|=1$, then $d$ is odd, otherwise $|\ct_4|\geq|\ct_3|\geq\frac{d}{2}+1$, it implies that $|\ct_3\cap\ct_4|\geq|\ct_3|+|\ct_4|-d\geq2$, contradiction. Hence $|\ct_3|=|\ct_4|=\frac{d+1}{2}$. Without lack of generality, we let $\ct_3\cap\ct_4=\{p_d\}$.

(9) If $d>3$, then $|\ct_1|\leq|\ct_2|\leq|\ct_3|\leq|\ct_4|=\frac{d+1}{2}\leq d-2$, then $1\leq|p_dx_i|_n\leq n-1$  for $i=1,2,3,4$.

(10) If $d=3$ and $\mu_d>1$, then $1\leq|p_dx_i|_n\leq n-1$  for $i=1,2,3,4$.

(11) If $d=3$ and $|\ct_2|\leq1$, then $1\leq|p_dx_i|_n\leq n-1$  for $i=1,2,3,4$.

(12) If $d=3$ and $|\ct_2|=2$, then $\ct_1$ is empty.

From all discussion above, we can assume that

$\ct_1=\emptyset, \ct_2=\{p_1,p_2\}, \ct_3=\{p_1,p_3\}, \ct_4=\{p_2,p_3\}$ and $\mu_3=1$. Since $|\ct_2|=|\ct_3|=|\ct_4|=2$, replace the positions of $x_2,x_3$ and repeat case (10), we can have  $\mu_2=1$. Similarly, $\mu_1=1$. Hence we have $\gcd(x_1,n)=1,\gcd(x_2,n)=p_1p_2,\gcd(x_3,n)=p_1p_3,\gcd(x_4,n)=p_2p_3.$  Up to now, Theorem 2.1 is proved.

If $S$ contains at least one $x_i$ coprime to $n$, then $u=\gcd(x_1,x_2,x_3,x_4,n)=1$. For $|\ct|<3$, Theorem 1.4 is proved by the results in [11] and [12]. Hence, in order to prove Theorem 1.4, it is sufficient to show the following Theorem 2.3:
\begin{theo}
Let $n=p_1p_2p_3$, where $p_1,p_2,p_3$ are three different primes, and $\gcd(n,6)=1$. Let $S=(x_1g)\cdot(x_2g)\cdot(x_3g)\cdot(x_4g)$ be a minimal zero-sum sequence over $G=\langle g\rangle$ such that $\ord(g)=n$, where
$(x_1,x_2,x_3,x_4)$ satisfies one of $(A2), (A3)$ and $(A4)$.
Then $\ind(S)=1$.
\end{theo}

Notice that, under each assumption of $(A2),(A3)$ and $(A4)$, we always assume that $(px_1g)\cdot(px_2g)\cdot(px_3g)\cdot(px_4g)$ is not a minimal zero-sum sequence for any $p\in\ct$.

\section{Preliminaries for Theorem 2.3}

Let $S$ be the sequence as described in Theorem 2.3.  Similar to Remark 2.1 of [11], we may always assume that
$x_1=1, 1+x_2+x_3+x_4=2n$ and $1<x_2<\frac{n}{2}<x_3\leq x_4<n-1$. Let $c=x_2,b=n-x_3,a=n-x_4$, then it is easy to show that the following proposition implies Theorem 2.3 under assumption $(A2),(A3)$ or $(A4)$.

\begin{proposition}
Let $n=p_1p_2p_3$, where $p_1,p_2,p_3$ are three different primes, and $\gcd(n,6)=1$. Let $S=(g)\cdot(cg)\cdot((n-b)g)\cdot((n-a)g)$ be a minimal zero-sum sequence over $G$ such that $\ord(g)=n$, where $1+c=a+b$, and

$(A2)$  $\{\gcd(c,n),\gcd(b,n), \gcd(a,n)\}=\{p_1, p_2, p_1p_2\}$.

$(A3)$  $\gcd(c+1,n)=p_1p_2$, $\gcd(b-1,n)=p_1p_3$, $\gcd(a-1,n)=p_2p_3$.

$(A4)$  $\gcd(c,n)=p_1p_2$, $\gcd(b,n)=p_1p_3$, $\gcd(a,n)=p_2p_3$.

{\noindent}Then $\ind(S)=1$.
\end{proposition}

\begin{remark}
Notice that $\gcd(n,6)=1$, and $(p_ig)\cdot(|p_ic|_ng)\cdot(|p_i(n-b)|_ng)\cdot(|p_i(n-a)|_ng)$ is not a minimal zero-sum sequence, we infer that $a\geq 36$ under $(A3)$, $a\geq 35$ under $(A2)$ or $(A4)$.
\end{remark}

\begin{lemma}
Proposition 3.1 holds if and only if one of the following conditions holds:

(1) There exist positive integers $k,m$ such that $\frac{kn}{c}\leq m\leq\frac{kn}{b}$, $\gcd(m,n)=1$, $1\leq k\leq b$, and $ma<n$.

(2) There exists a positive integer $M\in[1,\frac{n}{2}]$ such that $\gcd(M,n)=1$ and at least two of the following inequalities hold:
$$|Ma|_n>\frac{n}{2}, |Mb|_n>\frac{n}{2}, |Mc|_n<\frac{n}{2}.$$
\end{lemma}

\begin{lemma}
If there exist integers $k$ and $m$ such that $\frac{kn}{c}\leq m\leq\frac{kn}{b}$, $\gcd(m,n)=1$, $1\leq k\leq b$, and $a\leq\frac{kn}{b}$, then Proposition 3.1 holds.
\end{lemma}

From now on, we assume that $s=\lfloor\frac{b}{a}\rfloor$. Then we have $1\leq s\leq\frac{b}{a}<s+1$. Since $b\leq\frac{n}{2}$, we have $\frac{n}{2b}=\frac{(2s-t)n}{2b}-\frac{(2s-t-1)n}{2b}>1$, and then $[\frac{(2s-t-1)n}{2b},\frac{(2s-t)n}{2b}]$ contains at least one integer for every $t\in[0,\cdots,s-1]$.

\begin{lemma}
Suppose $s\geq2$ and $[\frac{(2s-2t-1)n}{2b},\frac{(s-t)n}{b}]$ contains an integer co-prime to $n$ for some $t\in[0,\cdots,\lfloor\frac{s}2\rfloor-1]$. Then Proposition 3.1 holds.
\end{lemma}

For the proof of Lemma 3.3, Lemma 3.4 and Lemma 3.5, one is referred to the proof of Lemma 2.3-2.5 in \cite{LP}, and we omit it here.

Let $\Omega$ denote the set of those integers: $x\in \Omega$ if and only if $x\in [\frac{(2s-t-1)n}{2b},\frac{(s-t)n}{b}]$ for some $t\in[0,\cdots,\lfloor\frac{s}2\rfloor-1]$. By Lemma 3.5, we also assume that

$(B)$: $[\frac{(2s-2t-1)n}{2b},\frac{(s-t)n}{b}]$ contains no integers co-prime to $n$ for every $t\in[0,\cdots,\lfloor\frac{s}2\rfloor-1]$.

\begin{lemma}
If $s\geq2$ and $[\frac{(2s-2t-1)n}{2b},\frac{(s-t)n}{b}]$ contains no integers co-prime to $n$ for every $t\in[0,\cdots,\lfloor\frac{s}2\rfloor-1]$. Then $[\frac{(2s-t-1)n}{2b},\frac{(s-t)n}{b}]$ contains at most 3 integers for every $t\in[0,\cdots,\lfloor\frac{s}2\rfloor-1]$. Hence $\frac{n}{2b}<4$.
\end{lemma}
\begin{proof}
If there exists some $t\in[0,\cdots,\lfloor\frac{s}2\rfloor-1]$ such that $[\frac{(2s-2t-1)n}{2b},\frac{(s-t)n}{b}]$ contains at least 4 integers, hence  $x,x+1,x+2,x+3\in\Omega$ for some $x$. It is easy to know that at least one of the four integers is co-prime to $n$. Then this lemma holds.
\end{proof}

\begin{lemma}
If $s\geq4$ and $[\frac{(2s-2t-1)n}{2b},\frac{(s-t)n}{b}]$ contains 3 integers for some $t\in[0,\cdots,\lfloor\frac{s}2\rfloor-1]$. Then one of the following  holds for some $x$:

(c1) $x,x+1,x+2,x+7,x+8,x+9\in\Omega$; (c2) $x,x+1,x+2,x+6,x+7,x+8\in\Omega$;

(c3) $x,x+1,x+2,x+5,x+6,x+7\in\Omega$; (c4) $x,x+1,x+2,x+6,x+7\in\Omega$;

(c5) $x,x+1,x+5,x+6,x+7\in\Omega$; (c6) $x,x+1,x+2,x+5,x+6\in\Omega$;

(c7) $x,x+1,x+4,x+5,x+6\in\Omega$.
\end{lemma}
\begin{proof}
Since $s\geq4$, we can consider $t=0,1$ respectively. Because $[\frac{(2s-2t-1)n}{2b},\frac{(s-t)n}{b}]$ contains exactly three integers for $t=0$ or $t=1$, we have $1\leq \frac{n}{2b}<3$, then (c1)-(c5) are all possible cases of the integers contained by $[\frac{(2s-2t-1)n}{2b},\frac{(s-t)n}{b}]$ for some $t=0,1$.
\end{proof}

\begin{lemma}
Suppose $s\geq4$ and $[\frac{(2s-2t-1)n}{2b},\frac{(s-t)n}{b}]$ contains no integers co-prime to $n$ for every $t\in[0,\cdots,\lfloor\frac{s}2\rfloor-1]$. Then $[\frac{(2s-2t-1)n}{2b},\frac{(s-t)n}{b}]$ contains at most two integers for every $t\in[0,\cdots,\lfloor\frac{s}2\rfloor-1]$ and $\frac{n}{2b}<3$.
\end{lemma}
\begin{proof}
If $[\frac{(2s-2t-1)n}{2b},\frac{(s-t)n}{b}]$ contains 3 integers for some $t\in[0,\cdots,\lfloor\frac{s}2\rfloor-1]$.

For (c1) case. Since $\gcd(x,n)>1, \gcd(x+1,n)>1, \gcd(x+2,n)>1, \gcd(x+7,n)>1, \gcd(x+8,n)>1, \gcd(x+9,n)>1$ and $\gcd(x,x+1,n)=\gcd(x,x+2,n)=\gcd(x+1,x+2,n)=1$, we have $\gcd(x+2,x+7)=5, \gcd(x+1,x+8)=7$, then $\gcd(x,x+9)>1, \gcd(x+1,x+9)=\gcd(x+2,x+9)=1$, so we have $\gcd(x+9,n)=1$, which contradicts to our assumption.

The proof for (c2)-(c7) is similar. Then this lemma holds.
\end{proof}

\begin{lemma}
If $s\geq 6$ and $[\frac{(2s-2t-1)n}{2b},\frac{(s-t)n}{b}]$ contains exactly two integers for every $t\in[0,\cdots,\lfloor\frac{s}2\rfloor-1]$. Then one of the following holds for some $x$:

(c8) $x,x+1,x+5,x+6,x+10,x+11\in\Omega$;

(c9) $x,x+1,x+4,x+5,x+9,x+10\in\Omega$;

(c10) $x,x+1,x+5,x+6,x+9,x+10\in\Omega$;

(c11) $x,x+1,x+4,x+5,x+8,x+9\in\Omega$;

(c12) $x,x+1,x+4,x+5,x+7,x+8\in\Omega$;

(c13) $x,x+1,x+3,x+4,x+7,x+8\in\Omega$;

(c14) $x,x+1,x+3,x+4,x+6,x+7\in\Omega$.
\end{lemma}
\begin{proof}
Similar to the proof of Lemma 3.7.
\end{proof}

\begin{lemma}
If $s\geq 6$, then there exist $t_1\in\{0,\cdots,\lfloor\frac{s}2\rfloor-1\}$ such that   $[\frac{(2s-t_1-1)n}{2b},\frac{(2s-t_1)n}{2b}]$ contains exactly one integer and $\frac{n}{2b}<2$.
\end{lemma}
\begin{proof}
Suppose that $[\frac{(2s-2t-1)n}{2b},\frac{(s-t)n}{b}]$ contains exactly two integers for every $t\in[0,\cdots,\lfloor\frac{s}2\rfloor-1]$.

Case (c8): if (c8) holds, then $\gcd(x,x+1,n)=\gcd(x+1,x+5,n)=\gcd(x,x+5,n)=1$ or $\gcd(x+1,x+5,n)=\gcd(x+5,x+6,n)=\gcd(x+1,x+6,n)=1$, $\gcd(x,x+5)=5$.

If $\gcd(x,x+1,n)=\gcd(x+1,x+5,n)=\gcd(x,x+5,n)=1$, then $\gcd(x,x+10)=\gcd(x+5,x+10)=1$ and $\gcd(x+1,x+10)>1$. However, $\gcd(x+1,x+10)\in\{1,3,9\}$, so $\gcd(x+10,n)=1$, which contradicts to our assumption.

If $\gcd(x+1,x+5,n)=\gcd(x+5,x+6,n)=\gcd(x+1,x+6,n)=1$, $\gcd(x,x+5)=5$, then $\gcd(x+1, x+11)=\gcd(x+6,x+11)=1$ and $\gcd(x+5,x+11)>1$. However, since $\gcd(x+5,x+11)|6$, we have $\gcd(x+11,n)=1$, contradiction.

The proof for (c9-c14) is similar. Then this lemma holds.
\end{proof}

\begin{lemma}
If $s\geq 8$ and there exists $t_2\in\{0,\cdots,\lfloor\frac{s}2\rfloor-1\}$ such that   $[\frac{(2s-t_2-1)n}{2b},\frac{(2s-t_2)n}{2b}]$ contains exactly two integers. Then one of the following holds for some $x$:

(c15) $x,x+3,x+4,x+7,x+8,x+11,x+12\in\Omega$;   (c16) $x,x+1,x+4,x+7,x+8,x+11,x+12\in\Omega$;

(c17) $x,x+1,x+4,x+5,x+8,x+11,x+12\in\Omega$;   (c18) $x,x+1,x+4,x+5,x+8,x+9,x+12\in\Omega$;

(c19) $x,x+1,x+3,x+4,x+7,x+10,x+11\in\Omega$;   (c20) $x,x+1,x+3,x+4,x+7,x+8,x+11\in\Omega$;

(c21) $x,x+3,x+4,x+6,x+7,x+10,x+11\in\Omega$;   (c22) $x,x+1,x+4,x+5,x+7,x+8,x+11\in\Omega$;

(c23) $x,x+3,x+4,x+7,x+8,x+10,x+11\in\Omega$;   (c24) $x,x+1,x+4,x+7,x+8,x+10,x+11\in\Omega$;

(c25) $x,x+1,x+3,x+4,x+6,x+7,x+10\in\Omega$; (c26) $x,x+3,x+4,x+6,x+7,x+9,x+10\in\Omega$;

(c27) $x,x+2,x+3,x+5,x+6,x+8,x+9\in\Omega$;   (c28) $x,x+1,x+3,x+5,x+6,x+8,x+9\in\Omega$;

(c29) $x,x+1,x+3,x+4,x+6,x+8,x+9\in\Omega$;   (c30) $x,x+1,x+3,x+4,x+6,x+7,x+9\in\Omega$;

(c31) $x,x+3,x+6,x+7,x+10,x+11\in\Omega$;   (c32) $x,x+3,x+6,x+7,x+9,x+10\in\Omega$;

(c33) $x,x+3,x+5,x+6,x+8,x+9\in\Omega$;   (c34) $x,x+2,x+4,x+5,x+7,x+8\in\Omega$;

(c35) $x,x+3,x+4,x+7,x+10,x+11\in\Omega$;   (c36) $x,x+2,x+3,x+5,x+7,x+8\in\Omega$;

(c37) $x,x+1,x+4,x+7,x+8,x+11\in\Omega$;  (c38) $x,x+1,x+3,x+5,x+6,x+8\in\Omega$;

(c39) $x,x+3,x+4,x+7,x+8,x+11\in\Omega$;  (c40) $x,x+2,x+3,x+5,x+6,x+8\in\Omega$;

(c41) $x,x+3,x+4,x+6,x+7,x+10\in\Omega$;  (c42) $x,x+1,x+3,x+6,x+8,x+9\in\Omega$;

(c43) $x,x+1,x+4,x+7,x+10,x+11\in\Omega$;  (c44) $x,x+1,x+3,x+5,x+7,x+8\in\Omega$;

(c45) $x,x+3,x+6,x+9,x+10\in\Omega$;  (c46) $x,x+3,x+6,x+8,x+9\in\Omega$;

(c47) $x,x+3,x+5,x+7,x+8\in\Omega$;  (c48) $x,x+2,x+5,x+7,x+8\in\Omega$;

(c49) $x,x+2,x+4,x+6,x+7\in\Omega$;  (c50) $x,x+1,x+4,x+7,x+10\in\Omega$;

(c51) $x,x+1,x+3,x+6,x+9\in\Omega$;  (c52) $x,x+1,x+3,x+5,x+8\in\Omega$;

(c53) $x,x+1,x+3,x+6,x+8\in\Omega$;  (c54) $x,x+1,x+3,x+5,x+7\in\Omega$;

(c55) $x,x+3,x+4,x+7,x+10\in\Omega$;  (c56) $x,x+2,x+3,x+5,x+8\in\Omega$;

(c57) $x,x+2,x+3,x+5,x+7\in\Omega$; (c58) $x,x+3,x+6,x+7,x+10\in\Omega$;

(c59) $x,x+3,x+5,x+6,x+8\in\Omega$; (c60) $x,x+2,x+4,x+5,x+7\in\Omega$.
\end{lemma}

\begin{lemma}
If $s\geq 8$, then $[\frac{(2s-2t-1)n}{2b},\frac{(2s-t)n}{b}]$ contains exactly one integer for every $t\in[0,\cdots,\lfloor\frac{s}2\rfloor-1]$.
\end{lemma}

\begin{lemma}
If $s\geq10$, then one of the following holds for some $x$:

(c61) $x,x+3,x+6,x+9,x+12\in\Omega$;  (c62) $x,x+2,x+5,x+8,x+11\in\Omega$;

(c63) $x,x+3,x+5,x+8,x+11\in\Omega$;  (c64) $x,x+3,x+6,x+8,x+11\in\Omega$;

(c65) $x,x+3,x+6,x+9,x+11\in\Omega$;  (c66) $x,x+2,x+4,x+7,x+10\in\Omega$;

(c67) $x,x+2,x+5,x+7,x+10\in\Omega$;  (c68) $x,x+2,x+5,x+8,x+10\in\Omega$;

(c69) $x,x+3,x+5,x+7,x+10\in\Omega$;  (c70) $x,x+3,x+5,x+8,x+10\in\Omega$;

(c71) $x,x+3,x+6,x+8,x+10\in\Omega$;  (c72) $x,x+2,x+4,x+6,x+9\in\Omega$;

(c73) $x,x+2,x+4,x+7,x+9\in\Omega$;  (c74) $x,x+2,x+5,x+7,x+9\in\Omega$;

(c75) $x,x+3,x+5,x+7,x+9\in\Omega$;  (c76) $x,x+2,x+4,x+6,x+8\in\Omega$.
\end{lemma}

\begin{lemma}
$s\leq9$.
\end{lemma}

The proof of Lemma 3.11 and Lemma 3.13 is similar to that of Lemma 3.7, and the proof of Lemma 3.12 and Lemma 3.14 is similar to that of Lemma 3.10.

In view of Lemmas 3.5 and 3.14,  from now on we may always assume that $s\leq9$.

Let $k_1$ be the largest positive integer such that $\lceil\frac{(k_1-1)n}{c}\rceil=\lceil\frac{(k_1-1)n}{b}\rceil$ and $\frac{k_1n}{c}\leq m<\frac{k_1n}{b}$. Since $\frac{bn}{c}\leq n-1<n=\frac{bn}{b}$ and $\frac{tn}{b}-\frac{tn}{c}=\frac{t(c-b)n}{bc}>2$ for all $t\geq b$, such integer $k_1$ always exists and $k_1\leq b$.

\begin{lemma}
Suppose $\lceil\frac{n}{c}\rceil=\lceil\frac{n}{b}\rceil$, then $k_1\leq\frac{b}{a}$.
\end{lemma}
\begin{proof}
Since $\lceil\frac{n}{c}\rceil=\lceil\frac{n}{b}\rceil$, we have $k_1\geq2$. Assume on the contrary that $k_1>\frac{b}{a}.$

Since $\lceil\frac{(k_1-1)n}{c}\rceil=\lceil\frac{(k_1-1)n}{b}\rceil$, we have
\beq 1>\frac{(k_1-1)n}{b}-\frac{(k_1-1)n}{c}=\frac{(c-b)(k_1-1)n}{cb}.\eeq

If $a-1\geq\frac{b}{k_1}$, then $\frac{(c-b)(k_1-1)n}{cb}=\frac{(a-1)(k_1-1)n}{cb}\geq\frac{(k_1-1)}{k_1}\times\frac{n}{c}>1$, contradiction. Thus we have that $\frac{b}{k_1}+1>a>\frac{b}{k_1}$ and $\frac{b}{k_1}$ is not an integer.

If  $k_1\geq 3$. Since $a\geq35$, we have  $\frac{(c-b)(k_1-1)n}{cb}=\frac{a-1}{a}\times\frac{a}{b}\times\frac{(k_1-1)n}{c}\geq\frac{34}{35}\times\frac{3-1}{3}\times2>1$, contradiction.

If $k_1=2$. $b<2a<b+2$, thus $b$ is an odd number, we may assume $b=2l+1$. Then $a=l+1$ and $c=3l+1$. (4.6) implies that $n<6l+5+\frac{1}{l}$. Moreover, $\gcd(n,6)=1$, by $6l+2<n\leq 6l+5$, we infer that  $n=6l+5$. Thus
\beqs \Big\lceil\frac{n}{c}\Big\rceil=\Big\lceil\frac{6l+5}{3l+1}\Big\rceil=3<4=\Big\lceil\frac{6l+5}{2l+1}\Big\rceil=\Big\lceil\frac{n}{b}\Big\rceil,\eeqs contradiction.
\end{proof}

Then we can show that Proposition 3.1 holds through the following two propositions.

\begin{proposition}
Suppose $\lceil\frac{n}{c}\rceil<\lceil\frac{n}{b}\rceil$, then Proposition 3.1 holds.
\end{proposition}

\begin{proposition}
Suppose $\lceil\frac{n}{c}\rceil=\lceil\frac{n}{b}\rceil$ and $k_1\leq\frac{b}{a}$, then Proposition 3.1 holds.
\end{proposition}

\section{Proof of Proposition 3.16}

\begin{lemma}
If $[\frac{n}{c},\frac{n}{b}]$ contains at least two integers, then $\ind(S)=1$.
\end{lemma}
\begin{proof}
We prove this lemma under assumption $(A4)$. For $(A2)$ and $(A3)$, the proof is very similar.

Since $a\leq b$, by Lemma 3.4 we may assume every integer in $[\frac{n}{c},\frac{n}{b}]$ in not co-prime to $n$. Since $n=p_1p_2p_3$, and $p_1p_3|b$, it follows that $\frac{n}{b}\leq p_2$. Then one of the following holds:

\beq m_1-1<\frac{n}{c}\leq m_1<m_1+1< \frac{n}{b}=m_1+2=p_2,\\
m_1-1<\frac{n}{c}\leq m_1<m_1+1\leq \frac{n}{b}<m_1+2.\eeq

{\it For case} (4.1): We have that $\gcd(m_1,n)>1, \gcd(m_1+1,n)>1$ and $\gcd(n,6)=1$, we infer that $p_2\geq23$ and $\gcd(n,m_1+3)=1$, then  $m_1\geq21$ and $n\geq23b$. Note that
\beq 2m_1-2<\frac{2n}{c}<\frac{2n}{b}=2m_1+4<2m_1+5.\eeq

Let $m=2m_1+5$ and $k=2$. Since  $1+c=a+b$, by (4.3) we have $(2m_1-2)(b+a-1)<(2m_1+5)b$, and thus $(2m_1-2)(a-1)<7b$. Since $a\geq 35$ and $m_1\geq 21$, we have
\beqs ma=(2m_1+5)a=\frac{2m_1+5}{2m_1-2}\times\frac{a}{a-1}\times(2m_1-2)(a-1)<\frac{2\times21+5}{2\times21-2}\times\frac{35}{35-1}\times7b<23b\leq n,\eeqs
and the result holds.

{\it For case} (4.2):
We have that $\gcd(m_1,n)>1, \gcd(m_1+1,n)>1$ and $\gcd(n,6)=1$, we infer that $m_1\geq10$. Then $n\geq 11b$.
Since $n=p_1p_2p_3$, we have $\gcd(2m_1+1,n)=1$ or $\gcd(2m_1+3,n)=1$. Note that
\beq 2m_1-2<\frac{2n}{c}\leq2m_1<2m_1+1<2m_1+2\leq\frac{2n}{b}<2m_1+4.\eeq

 Since  $1+c=a+b$, by (4.4) we have $(2m_1-2)(b+a-1)<(2m_1+4)b$, and thus $(2m_1-2)(a-1)<6b$. Let $k=2$ $m\in\{2m_1+1,2m_1+3\}$ be such that $\gcd(m,n)=1$.  Since $a\geq 35$ and $m_1\geq 10$, we have
\beqs ma\leq(2m_1+3)a=\frac{2m_1+3}{2m_1-2}\times\frac{a}{a-1}\times(2m_1-2)(a-1)<\frac{2\times10+3}{2\times10-2}\times\frac{35}{35-1}\times6b=\frac{805b}{102}<n,\eeqs
and the result holds.
\end{proof}

By Lemma 4.1, we nay assume that $[\frac{n}{c},\frac{n}{b}]$ contains exactly one integer $m_1$, and thus
\beqs m_1-1<\frac{n}{c}\leq m_1<\frac{n}{b}<m_1+1.\eeqs
Consequently, $\frac{n}{b}-m_1<1$ and $m_1-\frac{n}{c}<1$.

\begin{lemma}
If $4<\frac{n}{c}\leq5<\frac{n}{b}<6$ and $5|n$, then $\ind(S)=1$.
\end{lemma}
\begin{proof}
Since $4<\frac{n}{c}\leq5<\frac{n}{b}<6$, $n\geq5b$. Note that $m_1=\lceil\frac{n}{c}\rceil=5$. We divide the proof into eight cases.

{\noindent \bf Case 1.} $8<\frac{2n}{c}\leq9<10<\frac{2n}{b}<12$.

Since $8(b+a-1)=8c<2n<12b$, we have $8(a-1)<4b$. Clearly, $\gcd(n,9)=1$.  Then
\beqs 9a=\frac{9a}{8(a-1)}\times 8(a-1)<\frac{9}{8}\times\frac{35}{34}\times 4b=\frac{315b}{272}<5b\leq n.\eeqs
By Lemma 3.3(1), this lemma holds with $k=2, m=9$.

{\noindent \bf Case 2.} $9<\frac{2n}{c}\leq10<11<\frac{2n}{b}<12$ and $\gcd(n,11)=1$.

Since $9(b+a-1)=9c<12b$, we have $9(a-1)<3b$ and
\beqs 11a=\frac{11a}{9(a-1)}\times 9(a-1)<\frac{11}{9}\times\frac{35}{34}\times 3b=\frac{385b}{102}<5b\leq n.\eeqs
Then Lemma 3.3(1) can be applied with $k=2, m=11$.

{\noindent \bf Case 3.} $9<\frac{2n}{c}\leq10<11<\frac{2n}{b}<12$, $\gcd(n,11)=11$ and $13<\frac{3n}{c}\leq14<15<16<\frac{3n}{b}<18$.

It still holds  $9(a-1)<3b$. By assumption $n\geq1000$ and $5\times7\times11<1000$, we have $\gcd(n,14)=1$. Then
\beqs 14a=\frac{14a}{9(a-1)}\times 9(a-1)<\frac{14}{9}\times\frac{35}{34}\times 3b=\frac{245b}{51}<5b\leq n.\eeqs

{\noindent \bf Case 4.} $9<\frac{2n}{c}\leq10<11<\frac{2n}{b}<12$, $\gcd(n,11)=11$ and $14<\frac{3n}{c}\leq15<16<\frac{3n}{b}<18$.

In this case,  $14(a-1)<4b$. Since $\gcd(n,16)=1$, we have
\beqs 16a=\frac{16a}{14(a-1)}\times 14(a-1)<\frac{16}{14}\times\frac{35}{34}\times 4b=\frac{80b}{17}<5b\leq n.\eeqs

{\noindent \bf Case 5.} $9<\frac{2n}{c}\leq10<\frac{2n}{b}<11$.

In this case we have $9(a-1)<2b$.

{\it Subcase 5.1.}  $13<\frac{3n}{c}\leq15<16<\frac{3n}{b}<17$.

Clearly, $\gcd(n,16)=1$. Then
\beqs 16a=\frac{16a}{9(a-1)}\times 9(a-1)<\frac{16}{9}\times\frac{35}{34}\times 2b=\frac{560b}{153}<5b\leq n.\eeqs

{\it Subcase 5.2.} $13<\frac{3n}{c}\leq14<15<\frac{3n}{b}<16$ and $\gcd(n,7)=1$.

\beqs 14a=\frac{14a}{9(a-1)}\times 9(a-1)<\frac{14}{9}\times\frac{35}{34}\times 2b=\frac{490b}{153}<5b\leq n.\eeqs

{\it Subcase 5.3.} $13<\frac{3n}{c}\leq14<15<\frac{3n}{b}<16$ and $\gcd(n,7)>1$.

Since $n>1000$ and $35|n$, we have $\gcd(19,n)=1$. It is easy to know that $18<\frac{4n}{c}\leq19<20<\frac{3n}{b}<22$. Then
\beqs 19a=\frac{19a}{9(a-1)}\times 9(a-1)<\frac{19}{9}\times\frac{35}{34}\times 2b=\frac{665b}{153}<5b\leq n.\eeqs

{\noindent \bf Case 6.} $14<\frac{3n}{c}\leq15<\frac{3n}{b}<16$.

In this case we have $14(a-1)<2b$.

{\it Subcase 6.1.}  $18<\frac{4n}{c}\leq19<20<21<\frac{3n}{b}<22$.

By assumption $n>1000$ and $5\times7\times19=665<1000$, we have $\gcd(n,21)=1$ or $\gcd(n,19)=1$. Let $m$ be one of $19$ and $21$ such that $\gcd(n,m)=1$. Then
\beqs ma\leq21a=\frac{21a}{14(a-1)}\times 14(a-1)<\frac{21}{14}\times\frac{35}{34}\times 2b=\frac{105b}{34}<5b\leq n.\eeqs

{\it Subcase 6.2.} $19<\frac{4n}{c}\leq20<21<\frac{4n}{b}<22$ and $\gcd(n,21)=1$.

The proof is similar to above.

{\it Subcase 6.3.} $19<\frac{4n}{c}\leq20<21<\frac{4n}{b}<22$ and $\gcd(n,21)>1$.

In this case we have $\gcd(n,26)=1$ and  $23<\frac{5n}{c}\leq25<26<\frac{5n}{b}<28$. Then
\beqs 26a=\frac{26a}{14(a-1)}\times 14(a-1)<\frac{26}{14}\times\frac{35}{34}\times 2b=\frac{65b}{17}<5b\leq n.\eeqs

{\it Subcase 6.4.} $18<\frac{4n}{c}\leq19<20<\frac{3n}{b}<21$.

It must hold that $23<\frac{5n}{c}\leq24<25<\frac{3n}{b}<27$. Since $\gcd(24,n)=1$, we have
\beqs 24a=\frac{24a}{14(a-1)}\times 14(a-1)<\frac{24}{14}\times\frac{35}{34}\times 2b=\frac{60b}{17}<5b\leq n.\eeqs

{\noindent \bf Case 7.} $19<\frac{4n}{c}\leq20<\frac{4n}{b}<21$.

In this case, $19(a-1)<2b$.

{\it Subcase 7.1.} $23<\frac{5n}{c}\leq24<25<\frac{5n}{b}<27$.

Since $\gcd(24,n)=1$, we have
\beqs 24a=\frac{24a}{19(a-1)}\times 19(a-1)<\frac{24}{19}\times\frac{35}{34}\times 2b=\frac{840b}{323}<5b\leq n.\eeqs

{\it Subcase 7.2.} $24<\frac{5n}{c}\leq25<26\frac{5n}{b}<27$.

It must hold that $32<\frac{7n}{c}\leq35<36\frac{5n}{b}<39$. Since $\gcd(n,36)=1$, we have
\beqs 36a=\frac{36a}{19(a-1)}\times 19(a-1)<\frac{36}{19}\times\frac{35}{34}\times 2b=\frac{1260b}{323}<5b\leq n.\eeqs

{\noindent \bf Case 8.} $24<\frac{5n}{c}\leq25<\frac{5n}{b}<26$.

By directly computing, we have $24(a-1)<2b$, then $\frac{b}{a}>\frac{24}{2}\times\frac{33}{34}=\frac{198}{17}$ and $s\geq11$, which contradicts to our assumption $s\leq9$.
\end{proof}

\begin{lemma}
If $6<\frac{n}{c}\leq7<\frac{n}{b}<8$ and $7|n$, then $\ind(S)=1$.
\end{lemma}

\begin{proof}
Since $6<\frac{n}{c}\leq7<\frac{n}{b}<8$, $n\geq7b$. Note that $m_1=\lceil\frac{n}{c}\rceil=7$. We divide the proof into five cases.

{\noindent \bf Case 1.} $12<\frac{2n}{c}\leq13<14<15<\frac{2n}{b}<16$.

Since $13\times7\times5<1000$, we have $\gcd(13,n)=1$ or $\gcd(15,n)=1$. Let $m$ be one of $13$ and $15$ such that $\gcd(n,m)=1$. Easily, we have $12(a-1)<4b$, then
\beqs ma\leq15a=\frac{15a}{12(a-1)}\times 12(a-1)<\frac{15}{12}\times\frac{35}{34}\times 4b=\frac{175b}{34}<7b\leq n.\eeqs

{\noindent \bf Case 2.} $12<\frac{2n}{c}\leq13<14<\frac{2n}{b}<15$.

It must hold that $24<\frac{4n}{c}\leq26<27<28<\frac{4n}{b}<30$ and $12(a-1)<3b$. Since $\gcd(n,27)=1$,  we have
\beqs 27a=\frac{27a}{12(a-1)}\times 12(a-1)<\frac{27}{12}\times\frac{35}{34}\times 3b=\frac{945b}{136}<7b\leq n.\eeqs

{\noindent \bf Case 3.} $13<\frac{2n}{c}\leq14<15<\frac{2n}{b}<16$.

{\it Subcase 3.1.} $\gcd(n,15)=1$. The proof is similar to Case 1.

{\it Subcase 3.2.} $\gcd(n,15)=5$. Then we have $\gcd(n,22)=1$, $12(a-1)<3b$ and $19<\frac{3n}{c}\leq21<22<\frac{2n}{b}<24$.
\beqs 22a=\frac{22a}{12(a-1)}\times 12(a-1)<\frac{22}{12}\times\frac{35}{34}\times 3b=\frac{385b}{68}<7b\leq n.\eeqs

{\noindent \bf Case 4.} $13<\frac{2n}{c}\leq14<\frac{2n}{b}<15$.

In this case $13(a-1)<2b$.

{\it Subcase 4.1.} $19<\frac{3n}{c}\leq21<22<\frac{2n}{b}\leq23$.

We have $\frac{5n}{c}\leq35<36<\frac{5n}{b}$ and $\gcd(n,36)=1$. Then
\beqs 36a=\frac{36a}{13(a-1)}\times 13(a-1)<\frac{36}{13}\times\frac{35}{34}\times 2b=\frac{1260b}{221}<7b\leq n.\eeqs

{\it Subcase 4.2.} $19<\frac{3n}{c}\leq20<21<\frac{3n}{b}\leq22$ and $\gcd(n,20)=1$.

\beqs 20a=\frac{20a}{13(a-1)}\times 13(a-1)<\frac{20}{13}\times\frac{35}{34}\times 2b=\frac{700b}{221}<7b\leq n.\eeqs

{\it Subcase 4.3.} $19<\frac{3n}{c}\leq20<21<\frac{3n}{b}\leq22$ and $\gcd(n,20)=5$.

It must hold $\frac{4n}{c}\leq27<28<\frac{4n}{b}$.  Since  $\gcd(27,n)=1$, we have
\beqs 27a=\frac{27a}{13(a-1)}\times 13(a-1)<\frac{27}{13}\times\frac{35}{34}\times 2b=\frac{945b}{221}<7b\leq n.\eeqs

{\noindent \bf Case 5.} $20<\frac{3n}{c}\leq21<\frac{3n}{b}\leq22$.

If $29<\frac{4n}{b}$, the proof is similar to Subcase 4.1.  If $\frac{4n}{c}\leq 27$, the proof is similar to Subcase 4.2.

If $27<\frac{4n}{c}\leq28<\frac{4n}{b}\leq29$, we have $27(a-1)<2b$, then $\frac{b}{a}>\frac{27}{2}\times\frac{33}{34}=\frac{891}{68}$ and $s\geq13$, which contradicts to our assumption $s\leq9$.

\end{proof}

Let $l$ be the smallest integer such that $[\frac{ln}{c},\frac{ln}{b})$ contains at least four integers. Clearly, $l\geq3$.  Since $\frac{n}{b}-m_1<1$ and $m_1-\frac{n}{c}<1$, by using the minimality of $l$ we obtain that
$lm_1-4<\frac{ln}{c}<\frac{ln}{b}<lm_1+4$. Then $\frac{ln(c-b)}{bc}=\frac{ln}{b}-\frac{ln}{c}<(lm_1+4)-(lm_1-4)=8$ and thus
\beqs l<\frac{8bc}{(c-b)n}<\frac{8b}{(a-1)(m_1-1)}\leq \frac{8b}{(35-1)(5-1)}<b.\eeqs

We claim that $[\frac{ln}{c},\frac{ln}{b})$ contains at most six integers.  For any positive integer $j$,  let $N_j$ denote the number of integers contained in  $[\frac{jn}{c},\frac{jn}{b})$. Since \beqs \Big(\frac{(j+1)n}{b}-(j+1)m_1\Big)-\Big(\frac{jn}{b}-jm_1\Big)=\frac{n}{b}-m_1<1,\\
\Big((j+1)m_1-\frac{(j+1)n}{c}\Big)-\Big(jm_1-\frac{jn}{c}\Big)=m_1-\frac{n}{c}<1,\eeqs
we infer that $N_{j+1}-N_j\leq2$, it is sufficient to show our claim.

By the claim above we have
\beqs lm_1-j_0<\frac{ln}{c}\leq lm_1-j_0+1<\cdots<lm_1-j_0+4<\frac{ln}{b}\leq lm_1-j_0+6,\eeqs
for some  $1\leq j_0\leq 4$.
We remark that since $n=p_1p_2p_3$ and $[\frac{ln}{c},\frac{ln}{b})$ contains at least four integers, one of them (say $m$) must be co-prime to $n$. If $ma<n$, then we have done by Lemma 3.3(1)(with $k=l<b$).

\begin{lemma}
If $m_1\not=5,7$, then $\ind(S)=1$.
\end{lemma}
\begin{proof}
If $m_1\not=5,7$, then $m_1\geq10$ and $n\geq m_1b\geq10b$. Let $k=l$ and let $m$ be one of the integers in $[\frac{ln}{c},\frac{ln}{b})$ which is co-prime to $n$.

Then $(lm_1-j_0)(b+a-1)=(lm_1-j_0)c<ln\leq(lm_1-j_0+6)b$, so $(lm_1-j_0)(a-1)<6b$. Note that $m\leq lm_1+3$ and $l\geq 3$, then
\beqs ma\leq(lm_1+3)a=\frac{lm_1+3}{lm_1-j_0}\times\frac{a}{a-1}\times(lm_1-j_0)(a-1)\\
<\frac{3\times10+3}{3\times10-4}\times\frac{35}{34}\times6b=\frac{3465b}{442}<10b\leq n,\eeqs
and we have done.
\end{proof}

\section{Proof of Proposition 3.17}

In this section, we always assume that $\lceil\frac{n}{c}\rceil=\lceil\frac{n}{b}\rceil$, so $k_1\geq2$, and we also assume that $k_1\leq\frac{b}{a}$. Since $a\leq\frac{b}{k_1}$, by Lemma 3.3 we may assume that $\gcd(n,m_1)>1$ for every $m_1\in[\frac{k_1n}{c},\frac{k_1n}{b})$.

\begin{lemma}{\rm (Lemma 3.7 of \cite{LP})}

If $u<\frac{n}{c}<\frac{n}{b}<v$ for some real numbers $u,v$ and $u(k_1-1)>s+1$, then
\beqs n<\frac{uv(k_1-1)(s+1)}{u(k_1-1)-(s+1)}.\eeqs
\end{lemma}

\begin{lemma}
$k_1\leq6$.
\end{lemma}
\begin{proof}
If $k_1\geq7$, then $7\leq k_1\leq s\leq9$. By Lemma 3.10, $n<4b$. Applying Lemma 5.1 with $u=2$ and $v=4$, we infer that $n<240$, which contradicts to our assumption $n>1000$.
\end{proof}

\begin{lemma}
If $k_1=6$, then $\ind(S)=1$.
\end{lemma}
\begin{proof}
Assume that $k_1=6$. Then $s\geq6$, by Lemma 3.10, we have $n<4b$.  If $s\leq8$, applying Lemma 5.1 with $u=2$ and $v=4$, we infer that $n<360$, which contradicts to our assumption $n>1000$.

Let $s=9$.  If $3<\frac{n}{c}<\frac{n}{b}<4$, the proof is similar to above. If $2<\frac{n}{c}<\frac{n}{b}<3$, then $10<\frac{5n}{c}<\frac{5n}{b}<15$. By the definition  of $k_1$, we have $10+r<\frac{5n}{b}<\frac{5n}{b}\leq11+r$ for some $r=0,1,2,3,4$.

{\bf Case 1.} $r=0$.  $10<\frac{5n}{c}<\frac{5n}{b}\leq11$, then
\beqs 12<\frac{6n}{c}\leq13<\frac{6n}{b}<\frac{66}{5},&& 14<\frac{7n}{c}<15<\frac{7n}{b}<\frac{77}{5},\\
16<\frac{8n}{c}\leq17<\frac{8n}{b}<\frac{88}{5}, && 18<\frac{9n}{c}\leq19<\frac{9n}{b}<\frac{99}{5},\eeqs
and we can find  $m\in\{13,15,17,19\}$ such $\gcd(m,n)=1$, by Lemma 3.4, $\ind(S)=1$.

{\bf Case 2.} $r=1$.  $11<\frac{5n}{c}<\frac{5n}{b}<12$, then $\frac{77}{5}<\frac{7n}{c}<16<\frac{7n}{b}<\frac{84}{5}$ and $\gcd(16,n)=1$, by Lemma 3.4, $\ind(S)=1$.

{\bf Case 3.} $r=2$. $12<\frac{5n}{c}<\frac{5n}{b}<13$. If $\frac{84}{5}<\frac{7n}{c}<18<\frac{91}{5}$, then $\gcd(18,n)=1$, and  $\ind(S)=1$. Otherwise,
\beqs \frac{72}{5}<\frac{6n}{c}\leq15<\frac{6n}{b}<\frac{78}{5},&& \frac{84}{5}<\frac{7n}{c}\leq17<\frac{7n}{b}<18,\\
\frac{96}{5}<\frac{8n}{c}\leq20<\frac{8n}{b}<\frac{144}{7}, && \frac{108}{5}<\frac{9n}{c}<22<\frac{9n}{b}<\frac{162}{7},\eeqs
and we can find  $m\in\{5,11,17\}$ such $\gcd(m,n)=1$, by Lemma 3.4, $\ind(S)=1$. Otherwise, $n=5\times11\times17=935<1000$, which contradicts to our assumption.

{\bf Case 4.} $r=3$. $13<\frac{5n}{c}<\frac{5n}{b}<14$. Then $\frac{78}{5}<\frac{6n}{c}<16<\frac{6n}{b}<\frac{84}{5}$, and $\gcd(16,n)=1$, by Lemma 3.4, $\ind(S)=1$.

{\bf Case 5.} $r=4$. $14<\frac{5n}{c}<\frac{5n}{b}<15$. Then \beqs \frac{84}{5}<\frac{6n}{c}\leq17<\frac{6n}{b}<18,&& \frac{98}{5}<\frac{7n}{c}<20<\frac{7n}{b}<21,\\
\frac{112}{5}<\frac{8n}{c}\leq23<\frac{8n}{b}<24, && \frac{126}{5}<\frac{9n}{c}<26<\frac{9n}{b}<27,\eeqs
and we can find  $m\in\{5,13,17,23\}$ such $\gcd(m,n)=1$,by Lemma 3.4, $\ind(S)=1$.
\end{proof}

\begin{lemma}
If $k_1=5$, then $\ind(S)=1$.
\end{lemma}
\begin{proof}
Assume that $k_1=5$. Since $s\geq5$, we have $n<6b$. If  $s\leq6$ or $\frac{11}{4}<\frac{n}{c}$, we can get a contradiction by applying Lemma 5.1.

For $s=7,8,9$, let $2<\frac{n}{c}<\frac{n}{b}<3$. We have $8+r<\frac{4n}{c}<\frac{4n}{b}\leq9+r$ for some $r=0,1,2,3$. If $r=3$, then $\frac{11}{4}<\frac{n}{c}$, which has been solved.

{\bf Case 1.} $r=0$. Then
\beqs 10<\frac{5n}{c}\leq11<\frac{5n}{b}<\frac{45}{4}, 12<\frac{6n}{c}\leq13<\frac{6n}{b}<\frac{27}{2}, 14<\frac{7n}{c}<15<\frac{7n}{b}<\frac{63}{4},\eeqs
and we can find  $m\in\{11,13,15\}$ such $\gcd(m,n)=1$ because $11\times13\times5<1000$. By Lemma 3.4, $\ind(S)=1$.

{\bf Case 2.} $r=1$. Then $\frac{45}{4}<\frac{5n}{c}<12<\frac{5n}{b}<\frac{25}{2}$,
and  $\gcd(12,n)=1$. By Lemma 3.4, $\ind(S)=1$.

{\bf Case 3.} $r=2$. Then $15<\frac{6n}{c}<16<\frac{6n}{b}<\frac{33}{2}$, and  $\gcd(16,n)=1$. Let $m=16$ and $k=6$,  we have $m\cdot a=16\times(c+1-n)\leq16\times(\frac{2n-1}{5}-\frac{4n}{11}+1)=\frac{16\times(2n+44)}{55}<\frac{2n+44}{3}<n$.
By Lemma 3.4, $\ind(S)=1$.
\end{proof}

\begin{lemma}
If $k_1=4$, then $\ind(S)=1$.
\end{lemma}
\begin{proof}

Assume that $k_1=4$.  Since $s\geq4$, we have $n<6b$.

{\bf Case 1.} $s=4$, or $3<\frac{n}{c}$ and $s=5,6,7$, we can get a contradiction by applying Lemma 5.1.

{\bf Case 2.}  If $2<\frac{n}{c}<\frac{n}{b}<3$ and $s=5,6,7$. Then $6+r<\frac{3n}{c}<\frac{3n}{b}\leq7+r$ for some $r=0,1,2$.

{\it Subcase 2.1.} $r=0$. We have $8<\frac{4n}{c}<\frac{4n}{b}\leq\frac{28}{3}$, then $m_1=9$, which contradicts to $\gcd(n,m_1)=1$.

{\it Subcase 2.2.} $r=1$. We have $12<\frac{5n}{c}\leq13<\frac{5n}{b}<\frac{40}{3}$ or $\frac{35}{3}<\frac{5n}{c}<12<\frac{5n}{b}<\frac{40}{3}$.

If $\frac{5n}{c}<12<\frac{5n}{b}$, then $\ind(S)=1$ by Lemma 3.4. If $12<\frac{5n}{c}\leq13<\frac{5n}{b}<\frac{40}{3}$,
we have $\frac{9n}{2b}<12<13<\frac{5n}{b}$, and $12\in[\frac{9n}{2b},\frac{5n}{b}]$. Since $\gcd(n,12)=1$,  which contradicts to our previous assumption $(B)$ with $t=0,1,2$ for $s=5,6,7$, respectively.

{\it Subcase 2.3.} $r=2$. We have $\frac{8}{3}<\frac{n}{c}<\frac{n}{b}<3$, we can get a contradiction  by Lemma 5.1.

{\bf Case 3.} $s=8,9$. We have $n<4b$. Then $6+r<\frac{3n}{c}<\frac{3n}{b}\leq7+r$ for some $r=0,1,2,3,4,5$.

{\it Subcase 3.1.} $r=0$. We have $8<\frac{4n}{c}<\frac{4n}{b}<\frac{28}{3}$, then $m_1=9$, which contradicts to $\gcd(n,m_1)=1$.

{\it Subcase 3.2.} $r=1$. We have $\frac{28}{3}<\frac{4n}{c}\leq10<\frac{4n}{b}<\frac{32}{3}$. Assume that $5|n$, otherwise $\ind(S)=1$ by Lemma 3.4.   Furthermore, if $\frac{5n}{c}<12<\frac{5n}{b}$, we also have $\ind(S)=1$ by Lemma 3.4. Then  $12<\frac{5n}{c}\leq13<\frac{5n}{b}<\frac{40}{3}$.
Since $\gcd(n,18)=1$, we infer that $\frac{84}{5}<\frac{7n}{c}\leq17<\frac{7n}{b}<18$ and $n=5\times13\times17$. Otherwise, we have $\frac{13n}{2b}<\frac{52}{3}<18<\frac{7n}{b}$, which contradicts to $(B)$ with $t=1,2$ for $s=8,9$, respectively.

Under assumption $(A4)$:  $17\geq\frac{n}{a}=\frac{n}{b}\times\frac{b}{a}\geq2\times8=16>13>\frac{17}{2}$, we infer that $a=5\times13$. Because $8\leq\frac{b}{a}<10$,  and $b=j\times17\times5 (j<7)$ or $b=j\times 17\times 13 (j<3)$. However,
$\frac{6\times17\times5}{5\times13}=\frac{102}{13}<8$, $\frac{2\times17\times13}{5\times13}=\frac{34}{5}<8$, contradiction.

Under assumption $(A3)$:  we infer that $a=5\times13+1$. Because $8\leq\frac{b}{a}<10$,  and $b=j\times17\times5+1 (j<7)$ or $b=j\times 17\times 13+1 (j<3)$. However,
$\frac{6\times17\times5+1}{5\times13+1}=\frac{511}{66}<8$, $\frac{2\times17\times13+1}{5\times13+1}=\frac{443}{66}<8$, contradiction.

Under assumption $(A2)$, we distinguish three cases.

$(A2.1)$: $p_1p_2|a,p_1|b, p_2|c$, then $a=5\times13$. Moreover, $40\leq b<50$ when $p_1=5$ and $104\leq b<130$ when $p_1=13$. If $p_1=5$, then $\frac{n}{b}\geq \frac{5\times13\times17}{50}>22$,  contradiction. If $p_1=13$, then $\frac{n}{b}\geq \frac{5\times13\times17}{130}=\frac{85}{8}$,  contradiction.

$(A2.2)$: $p_1|a, p_1p_2|b, p_2|c$.  If $p_1=5$, then by $a=j\times5$ and $16<\frac{n}{a}\leq17$, we have $j=13$,  contradiction. If $p_1=13$, then by $a=j\times5$ and $16<\frac{n}{a}\leq17$, we have $j=5$, contradiction. If $p_1=17$, then $a=4\times17=68$. Moreover, $b=2\times13\times17$ or $b=j\times5\times17(j=4,5,6)$. If $b=2\times13\times17$, then $c=a+b-1=373$, which contradicts to $\gcd(n,c)=5$. If $b=j\times5\times17(j=4,5,6)$, then
$c\in\{261,356,431\}$, which contradicts to $\gcd(n,c)=13$.

$(A2.3)$: $p_1|a, p_2|b, p_1p_2|c$.  Similar to  $(A2.2)$, we have $a=4\times17=68$. Then $544=68\times8\leq b<680$. Since $b<\frac{n}{2}<553$, we have $p_2=5, b\in\{545,550\}$ or $p_2=13, b=546$. If $p_2=5$, then $c\in\{612,617\}$, which contradicts to $\gcd(n,c)=5\times17$. If $p_2=13$, then $c=613$, which contradicts to $\gcd(n,c)=17\times13$.

 {\it A remark on the proof.} From now on, throughout this section,  if $n$ is determined as a product of three small explicit primes similar to above, we only check it under assumption $(A4)$. The proof for $(A2)$ and  $(A3)$ is not essentially different from the above process.

{\it Subcase 3.3.} $r=2$. We have
\beqs  \frac{32}{3}<\frac{4n}{c}\leq11<\frac{4n}{b}<12, \frac{40}{3}<\frac{5n}{c}\leq14<\frac{5n}{b}<15, 16<\frac{6n}{c}\leq17<\frac{6n}{b}<18.\eeqs
Then $n=7\times11\times17$. Otherwise, there exists $m\in\{11,14,17\}$ such that $\gcd(n,m)=1$ and $\ind(S)=1$.
Clearly, $17\geq\frac{n}{a}=\frac{n}{b}\times\frac{b}{a}\geq2\times8=16>11>\frac{17}{2}$, we have $a=7\times 11$. Because $8\leq\frac{b}{a}<10$,  and $b=j\times17\times7 (j<6)$ or $b=j\times 17\times 11 (j<4)$. However,
$\frac{5\times17\times7}{7\times11}=\frac{85}{11}<8$, $\frac{3\times17\times13}{7\times13}=\frac{51}{7}<8$, contradiction.

{\it Subcase 3.4.} $r=3$. We have $15<\frac{5n}{c}<16<\frac{5n}{b}<\frac{50}{3}<17$ and $\gcd(n,16)=1$. Let $m=16$ and $k=5$,  we have $m\cdot a=16\times(c+1-n)\leq16\times(\frac{n-1}{3}-\frac{3n}{10}+1)=\frac{16\times(n+28)}{10}<n$,  then $\ind(S)=1$.

{\it Subcase 3.5.} $r=4$. We have
\beqs  \frac{40}{3}<\frac{4n}{c}\leq14<\frac{4n}{b}\leq\frac{44}{3},
\frac{50}{3}<\frac{5n}{c}<\frac{5n}{b}\leq\frac{55}{3}<19,
\frac{70}{3}<\frac{7n}{c}<\frac{7n}{b}\leq\frac{77}{3}<26.\eeqs
Since $\gcd(18,n)=\gcd(24,n)=1$, we infer that
$\frac{50}{3}<\frac{5n}{c}\leq17<\frac{5n}{b}<18$, $24<\frac{7n}{c}<25<\frac{7n}{b}<\frac{77}{3}<26$. Because $5\times7\times17<1000$, at least one of $14,17,25$ is co-prime to $n$.

{\it Subcase 3.6.} $r=5$. By Lemma 5.1, we infer that $n<1000$ with $u=\frac{11}{3}$ and $v=4$, contradiction.
\end{proof}

\begin{lemma}
If $k_1=3$, then $\ind(S)=1$.
\end{lemma}
\begin{proof}
We distinguish five cases.

{\bf Case 1.}  $s=3$.  Then $\frac{n}{b}<8$, and we have $4+r<\frac{2n}{c}<\frac{2n}{b}\leq5+r$ for some $r=0,1,2,\cdots,11$.

{\it Subcase 1.1.} $r\geq1$. We infer that $n<160$ with $u=\frac{5}{2}$ and $v=8$, contradiction.

{\it Subcase 1.2.} $r=0$. We have $8<\frac{4n}{c}<9<\frac{4n}{b}\leq10$, and $\gcd(9,n)=10$. Let $k=4$ and $m=9$, then $ma=9\times(c-b+1)\leq9\times(\frac{n-1}{2}-\frac{2n}{5}+1)=\frac{9n+45}{10}<n$, then $\ind(S)=1$ by Lemma 3.3(1).

{\bf Case 2.} $s=4$.  Then $\frac{n}{b}<6$, and we have $4+r<\frac{2n}{c}<\frac{2n}{b}\leq5+r$ for some $r=0,1,2,\cdots,7$.

{\it Subcase 2.1.} $r=0$. We have $8<\frac{4n}{c}<9<\frac{4n}{b}\leq10$, and $\gcd(n,9)=1$, $\ind(S)=1$.

{\it Subcase 2.2.} $r=1$. We have $\frac{15}{2}<\frac{3n}{c}<8<\frac{3n}{b}<9$, and $\gcd(n,8)=1$, $\ind(S)=1$.

{\it Subcase 2.3.} $r\geq2$. Then $3<\frac{n}{c}<\frac{n}{b}<6$, we infer that $n<180$, contradiction.

{\bf Case 3.} $s=5$.  Then $\frac{n}{b}<6$, and we have $4+r<\frac{2n}{c}<\frac{2n}{b}\leq5+r$ for some $r=0,1,2,\cdots,7$.

{\it Subcase 3.1.} $r=0$. We have $8<\frac{4n}{c}<9<\frac{4n}{b}\leq10$, and $\gcd(n,9)=1$, $\ind(S)=1$.

{\it Subcase 3.2.} $r=1$. We have $\frac{15}{2}<\frac{3n}{c}<8<\frac{3n}{b}\leq9$, and $\gcd(n,8)=1$, $\ind(S)=1$.

{\it Subcase 3.3.} $r=2$. We have $9<\frac{3n}{c}<10<\frac{3n}{b}\leq\frac{21}{2}$, $12<\frac{4n}{c}\leq13<\frac{4n}{b}\leq14$, $15<\frac{5n}{c}<\frac{5n}{b}\leq\frac{35}{2}$.  Since $\gcd(n,16)=1$, we infer that $16<\frac{5n}{c}\leq17<\frac{5n}{b}\leq\frac{35}{2}$ and $\frac{9n}{2b}<\frac{63}{4}<16<\frac{5n}{b}$, $\ind(S)=1$.

{\it Subcase 3.4.} $r\geq3$. Then $\frac{7}{2}<\frac{n}{c}<\frac{n}{b}<6$, we infer that $n<294$, contradiction.

{\bf Case 4.} $s=6,7,8$.  Then $\frac{n}{b}<4$, and we have $4+r<\frac{2n}{c}<\frac{2n}{b}\leq5+r$ for some $r=0,1,2,3$.

{\it Subcase 4.1.} $r=0$. We have $8<\frac{4n}{c}<9<\frac{4n}{b}\leq10$, and $\gcd(n,9)=1$, $\ind(S)=1$.

{\it Subcase 4.2.} $r=1$. We have $\frac{15}{2}<\frac{3n}{c}<8<\frac{3n}{b}\leq9$, and $\gcd(n,8)=1$, $\ind(S)=1$.

{\it Subcase 4.3.} $r=2$. Same to {\it Subcase 3.3.}

{\it Subcase 4.4.} $r=3$. We have $\frac{35}{2}<\frac{5n}{c}<\frac{5n}{b}\leq20$. Since $\gcd(n,18)=1$, we infer that $18<\frac{5n}{c}<19<\frac{5n}{b}<20$. Then $\frac{9n}{2b}<18<\frac{5n}{b}$, contradiction to $(B)$ with $t=1,2,3$ for $s=6,7,8$, respectively.

{\bf Case 5.} $s=9$.  Then $\frac{n}{b}<4$, and we have $4+r<\frac{2n}{c}<\frac{2n}{b}\leq5+r$ for some $r=0,1,2,3$.

For $r=0,1,2$. Same to Case 4. Then let $r=3$, and $7<\frac{2n}{c}<\frac{2n}{b}<8$.

{\it Subcase 5.1.} $21<\frac{6n}{c}\leq22<23<\frac{6n}{b}<24$, then $\frac{49}{2}<\frac{7n}{c}<25<26<\frac{7n}{b}<28$, at least one of $22,23,25,26$ is co-prime to $n$, hence $\ind(S)=1$.

{\it Subcase 5.2.} $21<\frac{6n}{c}\leq22<\frac{6n}{b}<23$, $\frac{49}{2}<\frac{7n}{c}<25<26<\frac{7n}{b}<\frac{161}{6}$. We infer that at least one of $22,25,26$ is co-prime to $n$ and $\ind(S)$. Otherwise, $n=11\times5\times13<1000$, contradiction.

{\it Subcase 5.3.} $21<\frac{6n}{c}\leq22<\frac{6n}{b}<23$, $\frac{49}{2}<\frac{7n}{c}<25<\frac{7n}{b}<26$. Then $28<\frac{8n}{c}<29<\frac{8n}{b}<\frac{208}{7}$, $\frac{63}{2}<\frac{9n}{c}<29<\frac{9n}{b}<\frac{234}{7}$. We infer that $n=11\times5\times29$. Otherwise, at least one of $22,25,29$ is co-prime to $n$ and $\ind(S)=1$.  So $29\geq\frac{n}{a}=\frac{n}{b}\times\frac{b}{a}>2\times9=18>11$. Since $\frac{29}{2}<18$, we have $a=11\times5$. Then
$b=j\times 29\times5 (j<6)$ or $b=j\times 29\times11 (j<3)$ and $9\leq\frac{b}{a}<10$. However,
$\frac{3\times 29\times5}{11\times5}=\frac{97}{11}<9$, $\frac{1\times 29\times11}{11\times5}=\frac{29}{5}<9$, we have $b\in\{580,725,638\}$. If $b=638$, then $c=a+b-1=692$, which contradicts to $\gcd(n,c)=5\times29=145$.
If $b\in\{580,725\}$, then $c\in\{634,779\}$, which contradicts to $\gcd(n,c)=29\times11=319$.

\end{proof}

\begin{lemma}
If $k_1=2$, $4<\frac{2n}{c}\leq5<\frac{2n}{b}<6$ and $a\leq \frac{b}{2}$, then $\ind(S)=1$.
\end{lemma}

\begin{proof}
Note that $m_1=5$ and $b\geq2a\geq70$. Since $\gcd(n,m_1)>1$ we have $5|n$.  By the definition of $k_1$, we conclude that $[\frac{k_2n}{c},\frac{k_2n}{b})$ contains at least one integer for each $k_2\geq k_1=2$. Note that $6<\frac{3n}{c}<\frac{3n}{b}<9$. We distinguish three cases.

{\bf Case 1.}   $7<\frac{3n}{c}<8<\frac{3n}{b}<9$. Then $\frac{n}{3}<b<\frac{3n}{8}\leq c<\frac{3n}{7}$.

Since $\gcd(n,8)=1$. Let $m=8$ and $k=3(<70\leq b)$. Then $ma=m(c-b+1)\leq8\times(\frac{3n-1}{7}-\frac{n+1}{3}+1)<n$, and we are done.

{\bf Case 2.}   $6<\frac{3n}{c}\leq7<\frac{3n}{b}<8$. Then $\frac{3n}{8}<b<\frac{2n}{5}<\frac{3n}{7}\leq c<\frac{n}{2}$.

If $\gcd(n,7)=1$, then let $m=7$ and $k=3$. Since $\frac{3n}{8}<b<c<\frac{n}{2}$, $ma=m(c-b+1)\leq7\times(\frac{n-1}{2}-\frac{3n+1}{8}+1)<n$, and we are done.

Next assume that $7|n$. Note that $8<\frac{4n}{c}\leq10<\frac{4n}{b}<12$.

If $9\not\in[\frac{4n}{c},\frac{4n}{b})$, then $\frac{4n}{c}\geq9$. Let $m=12$ and $k=5$. Since $\frac{5n}{c}\leq 7\times\frac{5}{3}<12<10\times\frac{5}{4}<\frac{5}{4}\times\frac{4n}{b}=\frac{5n}{b}$ and $\frac{3n}{8}<b<c<\frac{4n}{9}$, we have
$ma=m(c-b+1)\leq12\times(\frac{4n-1}{9}-\frac{3n+1}{8}+1)<n$, and we are done.

If $9\in[\frac{4n}{c},\frac{4n}{b})$, then $\frac{4n}{c}\geq9$ and thus $\frac{3n}{8}<b<\frac{2n}{5}<\frac{4n}{9}<c<\frac{4n}{9}$. So
\beqs 8n+\frac{n}{2}<\frac{69n}{8}<23b<\frac{46n}{5}<9n+\frac{n}{2}<10n<\frac{92n}{9}<23c<\frac{23n}{2}=11n+\frac{n}{2}.\eeqs

Note that $a=c-b+1\leq\frac{n+3}{8}$. If $a>\frac{n}{8}$, let $M=12$. We obtain that $|Ma|_n>\frac{n}{2}$ and $|Mb|_n>\frac{n}{2}$, and we are done. If $a<\frac{n}{9}$, let $m=9$ and $k=4$, we have $ma<n$, and we are done. Then $\frac{n}{9}<a<\frac{n}{8}$, and thus
\beqs 2n+\frac{n}{2}<\frac{23n}{9}<23a<\frac{23n}{8}<3n.\eeqs

If $23c\leq11n$, then $\frac{n}{9}<a=c-b+1\leq\frac{19n+57}{184}$, which implies that $n<40$, contradiction. So we must have $23c>11n$. Similarly, we can show that $23b<9n$. Moreover, we have $\gcd(n,23)=1$, otherwise $n=5\times7\times23=805<1000$, contradiction. Then $|23|_n+|23c|_n+|23(n-b)|_n+|23(n-a)|_n=n$ and we are done.

{\bf Case 3.}   $6<\frac{3n}{c}\leq7<8<\frac{3n}{b}<9$. Then  $\frac{n}{3}<b<\frac{3n}{8}<\frac{3n}{7}\leq c<\frac{n}{2}$.

Note that $a=c-b+1\leq\frac{n+1}{6}$. If $a>\frac{n}{6}$, we have $n<6a\leq n+1$ implies that $6a=n+1$, and $\gcd(n,n+1)=\gcd(n,6a)=\gcd(n,a)>1$, contradiction. Then $a<\frac{n}{6}$.

{\it Subcase 3.1.} $11|n$. Then $\gcd(n,7)=1$,  $\gcd(n,13)=1$ and $\gcd(n,17)=1$. Otherwise, $n\leq 5\times11\times17=935<1000$, contradiction.

We may assume that $a>\frac{n}{7}$. Otherwise, we can let $m=7$ and $k=3$, we have $ma<n$, so the lemma follows from Lemma 3.3(1). Then $\frac{3n}{2}<\frac{13n}{7}<13a<\frac{13n}{6}<\frac{5n}{2}<4n<\frac{13n}{3}<13b<\frac{39n}{8}<5n<\frac{11n}{2}<\frac{39n}{7}<13c<\frac{13}{2}$.

If $13c<6n$, then $\frac{n}{7}<a=c-b+1\leq\frac{5n+23}{39}$, so $n<41$, contradiction. Hence we must have that $13c>6n$, and then $|13c|_n<\frac{n}{2}$. If $13a<2n$ or $13b>\frac{9n}{2}$, then $|13a|_n>\frac{n}{2}$ or $|13b|_n>\frac{n}{2}$. Since $\gcd(n,13)=1$, the lemma follows from Lemma 3.3(2) with $M=13$.

Next assume that $13a>2n$ and $13b<\frac{9n}{2}$. Then $\frac{2n}{13}<a<b<\frac{9n}{26}$. Therefore,
\beqs \frac{5n}{2}<\frac{34n}{13}<17a<\frac{17n}{6}<3n<\frac{11n}{2}<\frac{17n}{3}<17b<\frac{153n}{26}<6n.\eeqs
We infer that $|17a|_n>\frac{n}{2}$ and $|17b|_n>\frac{n}{2}$. Since $\gcd(n,17)=1$, the lemma follows from Lemma 3.3(2) with $M=17$.

{\it Subcase 3.2.} $7|n$. Then $\gcd(n,11)=1$ and $\gcd(n,13)=1$.

As in {\it Subcase 3.1.} We may assume that $a>\frac{n}{8}$, and by a similar argument, we can complete the proof with $M=11$ or $M=13$.

{\it Subcase 3.3.} $\gcd(n,7)=\gcd(n,11)=1$.  See the proof of Subcase 3.1 of Lemma 3.10 in \cite{LP}.
\end{proof}

\begin{lemma}
If $k_1=2$, then $\ind(S)=1$.
\end{lemma}
\begin{proof}

Since $k_1=2$, we have $\lceil\frac{n}{c}\rceil=\lceil\frac{n}{b}\rceil$ and  have $2+r<\frac{n}{c}<\frac{n}{b}\leq3+r$ for some $r=0,1,2,3,4,5$. If $t=2$, then $8<\frac{2n}{c}<9<\frac{2n}{b}<10$ and $\gcd(n,9)=1$, contradiction. By Lemma 4.14, we only need to prove it for $t\not=0,2$. Particularly, when $s\geq6$, $\frac{n}{b}<4$, we only need consider $r=1$.  We distinguish six cases.

{\bf Case 1.}  $s=2$.

{\it Subcase 1.1.} $r=1$. Then $\frac{3n}{2b}<6<\frac{2n}{b}$, we have $6\in[\frac{3n}{2b},\frac{2n}{b}]$ and $\gcd(n,6)=1$, contradiction.

{\it Subcase 1.2.} $r\geq3$. By Lemma 5.1, we infer that $n<60$ with $u=5,v=8$, contradiction.

{\bf Case 2.}  $s=3$.

{\it Subcase 2.1.} $r=1$. We infer that $6<\frac{2n}{c}\leq7<\frac{2n}{b}<8$ and $7|n$. Then
$9<\frac{3n}{c}<10<\frac{3n}{b}\leq11$, or $10<\frac{3n}{c}\leq11<\frac{3n}{b}<12$. Otherwise, $n=5\times7\times11<1000$, contradiction.

Then the proof is very similar to that in \cite{LP}.

{\it Subcase 2.2.} $r\geq3$. By Lemma 5.1, we infer that $n<320$ with $u=5,v=8$, contradiction.

{\bf Case 3.} $s=4$.  Then $\frac{n}{b}<6$, and $t\leq3$.

{\it Subcase 3.1.} $r=1$. We infer that $6<\frac{2n}{c}\leq7<\frac{2n}{b}<8$ and $7|n$. Then
$9<\frac{3n}{c}<10<\frac{3n}{b}\leq11$, or $10<\frac{3n}{c}\leq11<\frac{3n}{b}<12$. Otherwise, $n=5\times7\times11<1000$, contradiction.

If $9<\frac{3n}{c}<10<\frac{3n}{b}\leq11$, then $\gcd(n,13)=1$, otherwise $n=5\times7\times13<1000$. Hence we infer that $13<\frac{4n}{c}\leq14<\frac{4n}{b}\leq\frac{44}{3}$, and $\frac{4n}{b}>14>13>\frac{77}{6}>\frac{7n}{2b}$, contradiction.

If $10<\frac{3n}{c}\leq11<\frac{3n}{b}<12$, then $\gcd(n,5)=1$, otherwise $n=5\times7\times11<1000$. Hence we infer that $\frac{3n}{b}>11>10>\frac{5n}{2b}$, contradiction.

{\it Subcase 3.2.}  $r=3$. Then $10<\frac{2n}{c}\leq11<\frac{2n}{b}<12$, and $\frac{3n}{b}>\frac{33}{2}>16>15>\frac{5n}{2b}$. Then  $\gcd(n,16)=1$ and $16\in[\frac{5n}{2b},\frac{3n}{b}]$, contradiction.

{\bf Case 4.} $s=5$.  Then $\frac{n}{b}<6$, and $t\leq3$.

{\it Subcase 4.1.} $r=1$. We have $6<\frac{2n}{c}\leq7<\frac{2n}{b}<8$, and $7|n$.
If $9<\frac{3n}{c}<10<\frac{3n}{b}\leq11$, the proof is similar to {\it Subcase 3.1.}

If $10<\frac{3n}{c}\leq11<\frac{3n}{b}<12$, then $\gcd(n,5)=1$, otherwise $n=5\times7\times11<1000$. We infer that $\frac{40}{3}<\frac{4n}{c}\leq14<\frac{4n}{b}<15$ and $n=7\times11\times 17$. Moreover, $\frac{n}{a}=\frac{n}{b}\times\frac{b}{a}>10$ implies that $a=7\times11$ or $a=7\times17$.

If $a=7\times11$, then $b=j\times11\times17$ or $b=j\times7\times17$ for some $j$. By $s=5$, we have $\frac{b}{a}\in[5,6)$, and we can't find such $j$.

If $a=7\times17$, then $b=j\times11\times17$ or $b=j\times11\times7$ for some $j$. We infer that $b=9\times11\times7=693$ and $c=a+b-1=811$, which contradicts  to $\gcd(n,c)=11\times17$.

{\it Subcase 4.2.} $r=3$. Then $10<\frac{2n}{c}\leq11<\frac{2n}{b}<12$. Since $\gcd(n,16)=1$, We infer that  $16<\frac{3n}{c}\leq17<\frac{3n}{b}<18$ and we can assume that $11\times17|n$. Then $\frac{4n}{b}>\frac{17\times4}{3}>22>21>\frac{7n}{2b}$ and $n=7\times11\times17$. Similar to {\it Subcase 4.1.}, it is impossible.

{\bf Case 5.} $s=6$.

$r=1$. We have $6<\frac{2n}{c}\leq7<\frac{2n}{b}<8$, and $7|n$.
If $9<\frac{3n}{c}<10<\frac{3n}{b}\leq11$, the proof is similar to {\it Subcase 3.1.}

If $10<\frac{3n}{c}\leq11<\frac{3n}{b}<12$, similar to {\it Subcase 3.1.}, $n=7\times11\times 17$. More over, $\frac{n}{a}=\frac{n}{b}\times\frac{b}{a}>12$ implies that $a=7\times11$. Then $b=j\times11\times17$ or $b=j\times7\times17$ for some $j$. By $s=6$, $\frac{b}{a}\in[6,7)$, we infer that $b=4\times7\times17$, and $c=a+b-1=552$, which contradicts  to $\gcd(n,c)=11\times17$.

{\bf Case 6.} $s=7,8,9$.

$r=1$. We infer that $6<\frac{2n}{c}\leq7<\frac{2n}{b}<8$ and $7|n$. Then
$9<\frac{3n}{c}<10<\frac{3n}{b}\leq11$, or $10<\frac{3n}{c}\leq11<\frac{3n}{b}<12$. Otherwise, $n=5\times7\times11<1000$, contradiction.

{\it Subcase 6.1.} $9<\frac{3n}{c}<10<\frac{3n}{b}\leq11$.  Then $\gcd(n,13)=1$, otherwise $n=5\times7\times13<1000$. Hence we infer that $13<\frac{4n}{c}\leq14<\frac{4n}{b}\leq\frac{44}{3}$, and $\frac{7n}{b}>24>\frac{13n}{2b}$, contradiction.

{\it Subcase 6.2.} $10<\frac{3n}{c}\leq11<\frac{3n}{b}<12$.

We have $\frac{50}{3}<\frac{5n}{c}<14<\frac{5n}{b}<20$. If $\frac{5n}{c}<17<\frac{5n}{b}$, then $n=7\times11\times17$. More over, $\frac{n}{a}=\frac{n}{b}\times\frac{b}{a}>14$ and $\frac{19}{2}<14$ implies that $a=7\times11$. Then $b=j\times11\times17(j<4)$ or $b=j\times7\times17(j<6)$ for some $j$. We can't find suitable $j$ for $s=8,9$. When $s=7$, we have
$b=3\times11\times17$ or $b=5\times7\times17$.  If $b=3\times11\times17$, $c=a+b-1=637=7\times91$, which contradicts to $\gcd(n,c)=7\times17$. If $b=5\times7\times17$, $c=a+b-1=594=11\times54$, which contradicts to $\gcd(n,c)=11\times17$.

If $\frac{5n}{c}<18<\frac{5n}{b}$, then $\ind(S)=1$.

If $\frac{5n}{c}\leq19<\frac{5n}{b}$, then $n=7\times11\times19$.

More over, $\frac{n}{a}=\frac{n}{b}\times\frac{b}{a}>14$  and $\frac{19}{2}<14$ implies that $a=7\times11$. Then $b=j\times11\times19(j<4)$ or $b=j\times7\times19(j<6)$ for some $j$. We can't find suitable $j$ for $s=7,9$. When $s=8$, we have
$b=3\times11\times19$ or $b=5\times7\times19$.  If $b=3\times11\times19$, $c=a+b-1=703=19\times37$, which contradicts to $\gcd(n,c)=7\times19$. If $b=5\times7\times19$, $c=a+b-1=741=19\times39$, which contradicts to $\gcd(n,c)=11\times19$.

\end{proof}

{\noindent\bf Acknowledgements}

The author is thankful to the referees for valuable suggestions and to prof. Yuanlin Li and prof. Jiangtao Peng for their useful discussion and valuable comments.

\vskip30pt
\def\refname{

\end{document}